\documentclass[reqno]{amsart}
\usepackage{amssymb,dsfont,graphicx,amscd,color,amsmath,amsfonts,amssymb,geometry,mathtools,paralist}
\usepackage[initials]{amsrefs}

\newtheorem{theorem}{Theorem}[section]
\newtheorem{proposition}[theorem]{Proposition}
\newtheorem{corollary}[theorem]{Corollary}
\newtheorem{lemma}[theorem]{Lemma}
\theoremstyle{remark}

\newtheorem{remark}[theorem]{Remark}

\theoremstyle{definition}
\newtheorem{definition}[theorem]{Definition}
\newtheorem{question}[theorem]{Question}
\DeclareMathOperator*{\SZ}{\mathcal{SZ}}

\newcommand{\R}{\mathbb R}

\newcommand{\C}{\mathbb C}

\DeclareMathOperator{\dom}{D}

\newcommand{\jones}{\mathcal{J}(\mathbb R)}
\DeclareMathOperator{\es}{\mathcal{ES}(\mathbb R)}

\DeclareMathOperator{\pes}{\mathcal{PES}(\mathbb R)}

\author[A. Bartoszewicz \and S. G\l \c ab \and D. Pellegrino \and J. B. Seoane-Sep\'{u}lveda]{Artur Bartoszewicz\textsuperscript{*} \and Szymon G\l \c ab\textsuperscript{*} \and Daniel Pellegrino\and Juan B. Seoane-Sep\'{u}lveda\textsuperscript{**}}
\address{Institute of Mathematics, \newline\indent Technical University of \L\'od\'z, \newline\indent W\'olcza\'nska 215, \newline\indent 93-005 \L\'od\'z, Poland}
\email {arturbar@p.lodz.pl}
\address{Institute of Mathematics, \newline\indent Technical University of \L\'od\'z, \newline\indent W\'olcza\'nska 215, \newline\indent 93-005 \L\'od\'z, Poland}
\email {szymon.glab@p.lodz.pl}
\address{Departamento de Matem\'{a}tica, \newline\indent Universidade Federal da Para\'{\i}ba, \newline\indent 58.051-900 - Jo\~{a}o Pessoa, Brazil.}\email{pellegrino@pq.cnpq.br}
\address{Departamento de An\'{a}lisis Matem\'{a}tico,\newline\indent Facultad de Ciencias Matem\'{a}ticas, \newline\indent Plaza de Ciencias 3, \newline\indent Universidad Complutense de Madrid,\newline\indent Madrid, 28040, Spain.}
\email{jseoane@mat.ucm.es}

\thanks{\textsuperscript{*} Supported by the Polish Ministry of Science and Higher Education Grant No.  N N201 414939 (2010-2013).}
\thanks{\textsuperscript{**} Supported by the Spanish Ministry of Science and Innovation, grant MTM2009-07848.}

\title[Algebrability, non-linear properties, and special functions]{Algebrability, non-linear properties, and special functions}
\subjclass[2010]{Primary: 15A03; Secondary: 28A20, 46J10}
\keywords{lineability; spaceability; algebrability; perfectly everywhere surjective functions; Sierpi\'nski-Zygmund functions}
\date{}
\begin{document}
\begin{abstract}
We construct uncountably generated algebras inside the following sets of \emph{special} functions:
\begin{inparaenum}
\item Sierpi\'nski-Zygmund functions.
\item Perfectly everywhere surjective functions.
\item Nowhere continuous Darboux functions.
\end{inparaenum}
All conclusions obtained in this paper are improvements of some already known results.
\end{abstract}
\maketitle
\section{Preliminaries and background}
This paper is a contribution to the very recent trend, in mathematical analysis, of the search for large algebraic structures (linear spaces or algebras) enjoying what one could call ``{\em special}'' properties. As it has become a usual notion nowadays, given a certain property we say that the subset $M$ of a topological vector space $X$ which satisfies it is \emph{$\mu$-lineable} (respectively, \emph{$\mu$-spaceable}) if $M \cup \{0\}$ contains a vector space (respectively, closed vector space) of dimension $\mu$ (finite or infinite). If $M$ contains an infinite-dimensional (closed) vector space, it shall be simply called \emph{lineable} (\emph{respectively} spaceable) for short. These notions of lineability and algebrability were coined by V. I. Gurariy in the early 2000's and first introduced in \cite{AGS,juanksu}.

In the recent years many examples of vector spaces of functions on $\mathbb{R}$ or $\mathbb{C}$ enjoying certain special properties have been constructed. Although these latter notions of lineability and spaceability did not appear until recently, the origins of this theory date back to 1966, when a famous example due to V. I. Gurariy (\cite{G1,G2}) showed that there exists an infinite dimensional linear space every non-zero element of which is a continuous nowhere differentiable function on $\mathcal{C}[0,1]$. (This result was later improved, and even Banach spaces of such functions were constructed.) More recently, many authors got interested in this subject and gave a wide range of examples. For instance, in  \cite{AGS} it was shown that the set of everywhere surjective functions in $\mathbb{R}$ is $2^\mathfrak{c}$-lineable (where $\mathfrak{c}$ denotes the cardinality of $\mathbb R$) and that the set of differentiable functions on $\mathbb R$ which are nowhere monotone is lineable in $\mathcal C(\mathbb R)$. These behaviors occur, sometimes, in particularly interesting ways. For example, in \cite{H}, Hencl showed that any separable Banach space is isometrically isomorphic to a subspace of $\mathcal C[0,1]$ whose non-zero elements are nowhere approximately differentiable and nowhere H{\"o}lder. We refer the interested reader to \cite{AGM,topology,BQ,B,BDP,BG2,BGP_LAA,BMP,GMPS_LAA,GMS_add,GGMS_MN,MPPS_LAA} for a wider range of results in this topic of lineability and spaceability.

Of course, one could go further and not just consider linear spaces but, instead, {\em larger} or {\em more complex} structures. For instance, in \cite{ACPS} the authors showed that there exists and uncountably generated algebra every non-zero element of which is an everywhere surjective function on $\mathbb{C}$, and in \cite{APS} it was shown that, if $E\subset\mathbb T$ is a set of measure zero, and if $\mathcal F(\mathbb T)$ denotes the subset of $\mathcal C(\mathbb T)$ of continuous functions whose Fourier series expansion diverges at every point of $E$, then $\mathcal F(\mathbb T)$ contains an infinitely generated and dense subalgebra. These previous examples motivated the notion of {\em algebrability}. This notion was first introduced in \cite{APS,AS} and here we give an slightly simplified version of it, that shall be sufficient for our purposes.
\begin{definition}[\cite{APS,AS}]
Let $\mathcal{L}$ be an algebra. A set $A \subset \mathcal{L}$ is said to be $\beta$-algebrable if there exists an algebra
$\mathcal{B}$ so that $\mathcal{B} \subset A \cup \{ 0\}$ and card$(Z) = \beta$, where  $\beta$ is a cardinal number and $Z$ is a
minimal system of generators of $\mathcal{B}$. Here, by {\em $Z = \{z_{\alpha}: \alpha \in \Lambda\}$ is a minimal system of generators
of $\mathcal{B},$} we mean that $\mathcal{B} = \mathcal{A}(Z)$ is the algebra generated by $Z,$ and for every $\alpha_0 \in
\Lambda,$\; $z_{\alpha_0} \notin \mathcal{A}(Z \setminus \{z_{\alpha_0}\}).$  We also say that $A$ is algebrable if $A$ is $\beta$-algebrable for $\beta$ infinite.
\end{definition}
\begin{remark}
\begin{enumerate}
\item Notice that if $Z$ is a minimal infinite system of generators of $\mathcal{B},$ then $\mathcal{A}(Z')\neq \mathcal{B}$ for any $Z'\subset \mathcal{B}$ such that card$(Z')<$ card$(Z)$. The result is not true for finite systems of generators:
Take  $X =\mathbb{C}^2$ with coordinate-wise multiplication. $X$ is a Banach algebra with unit $(1,1)$. The set $\{(1,0),(0,1)\}$ is a minimal system of generators of $X$. However, $X$ is also single generated by $u=(1,i)$: Consider $P:X \to X, P(s,t) = (s^2,t^2)$. Note that $P(u) = (1,-1)$ and so we get $\frac{1}{1+i}(u - P(u)) = (0,1) \in X.$ Similarly, we also have $(1,0) \in X.$
\item Of course, notice that algebrability implies lineability, and the converse (in general) does not hold. For instance (and as it was proved in \cite{Taiwan}), given any unbounded interval $I$, the set of Riemann-integrable functions on $I$ that are not Lebesgue-integrable is lineable and not algebrable.
\end{enumerate}
\end{remark}

There also exists an strengthening of the notion of algebrability (see \cite{BG1}). Given a cardinality $\kappa$, we say that $A$ is a $\kappa$-generated free algebra if there exists a subset $X = \{x_{\alpha} : \alpha < \kappa\}$ of $A$ such that any function $f$ from $X$ to some algebra $A'$ can be uniquely extended to a homomorphism from $A$ into $A'$. Then $X$ is called a set of free generators of the algebra $A$. A subset $X = \{x_{\alpha} : \alpha < \kappa\}$ of a commutative algebra $B$ generates a free sub-algebra $A$ if and only if for each polynomial $P$ and any $x_{\alpha_1}, x_{\alpha_2}, \ldots, x_{\alpha_n}$ we have $P(x_{\alpha_1}, x_{\alpha_2}, \ldots, x_{\alpha_n}) = 0$ if and only if $P=0$. Also, let us recall that $X = \{x_{\alpha}: \alpha < \kappa\} \subset E$ is a set of free generators of a free algebra $A \subset E$ if and only if the set $\widehat{X}$ of elements of the form $x_{\alpha_1}^{k_1} x_{\alpha_2}^{k_2} \cdots x_{\alpha_n}^{k_n}$ is linearly independent and all linear combinations of elements from $\widehat{X}$ are in $E \cup \{0\}$. With this at hand, let us recall the following definition (see \cite{BG1}).
\begin{definition}
We say that a subset $E$ of a commutative linear algebra $B$ is strongly $\kappa$-algebrable if there exists a $\kappa$-generated free algebra $A$ contained in $E \cup \{0\}$. 
\end{definition}

In general, there are subsets of linear algebras which are algebrable but not strongly algebrable (for instance, $c_{00}$ is algebrable in $c_0$ but it is not strongly $1$-algebrable, as it is shown in \cite{BG1}).

\noindent This paper is mainly devoted to the thorough study of the following recently considered classes of functions:
\begin{itemize}
\item Sierpi\'nski-Zygmund functions.
\item Perfectly everywhere surjective functions.
\item Nowhere continuous Darboux functions (nowhere continuous functions mapping connected sets to connected sets).
\end{itemize}

These latter classes have been lately considered in \cite{G,GMSS,GMS_add,monthly}, where the authors proved the existence of infinite dimensional linear spaces inside each of the previous classes. In particular, in \cite{GMSS}*{Theorem 5.6} it was proved that the set of Sierpi\'nski-Zygmund functions is $\mathfrak{c}^+$-lineable. In \cite{GMSS}*{Theorem 2.6} the authors showed the $2^\mathfrak{c}$-lineability of the set of perfectly everywhere surjective functions, and in \cite{GMS_add,monthly} similar results were obtained on the class of nowhere continuous functions mapping connected sets to connected sets. The aim of this paper is to improve all the previously mentioned results by obtaining algebrability and strong algebrability where just lineability was previously reached. In particular, and among other results, we shall prove that there exists an uncountably generated algebra every non-zero element of which is a nowhere continuous Darboux function (Theorem \ref{3th}) and that both, the sets of Sierpi\'nski-Zygmund functions (Theorem \ref{SZth}) and the set of perfectly everywhere surjective functions (Theorem \ref{PESth}), are (respectively) strongly $\mathfrak{c}^+$-algebrable and strongly $2^{\mathfrak{c}}$-algebrable.

We shall use standard set theoretical notation. As usual, $\omega$ shall denote the cardinality of $\mathbb{N}$ and $\mathfrak{c}$ shall denote the continuum. Also, we shall identify each cardinal number with the first ordinal of the same cardinality. Then a cardinal $\kappa$ is equal to the set of all ordinals less than $\kappa$, denoted $\kappa=\{\xi:\xi<\kappa\}$. The rest of the notation shall be rather usual. Classical Real Analysis and Abstract Algebra techniques together with Set Theory and Cardinal Theory shall be used throughout the paper.

\section{Sierpi\'nski-Zygmund functions}
This section shall be devoted to a type of function that first appeared in the early 1920's (see \cite{SZ}) and due to W. Sierpi\'nski and A. Zygmund. A clear consequence of the classic Luzin's Theorem is that, for every measurable function \(f\colon\mathbb R\to\mathbb R\), there exists a measurable set \(S\subset\mathbb R\), of infinite measure, such that \(f\vert_S\) is continuous. On the other hand, and given any arbitrary function \(f\colon\mathbb R\to\mathbb R\), is it possible to find a dense set \(S\subset\mathbb R\) such that \(f\vert_S\) is continuous? The answer to this question was given by Blumberg (\cite{Blumberg}, 1922):

\begin{theorem}[Blumberg]
Let \(f\colon\mathbb R\to\mathbb R\) be an arbitrary function. There exists a dense subset \(S\subset\mathbb R\) such that the function \(f\vert_S\) is continuous.
\end{theorem}

Checking the proof of Blumberg's theorem (see, \cite{Kharazishvili}*{p.\ 154}) it can be noticed that the set $S$ above is countable. Naturally, we could wonder whether we can choose the set \(S\) in Blumberg's theorem to be uncountable. This was partially answered (in the negative) by W. Sierpi\'nski and A. Zygmund a year later in \cite{SZ} (see also \cite{Kharazishvili}*{pp. 165-166} for a modern reference).

\begin{theorem}[Sierpi\'{n}ski-Zygmund]
There exists a function \(f\colon\mathbb R\to\mathbb R\) such that, for any set \(Z\subset\mathbb R\) of cardinality the continuum, the restriction \(f\vert_Z\) is not a Borel map (and, in particular, not continuous).
\end{theorem}

Notice that, if the Continuum Hypothesis holds, the restriction of this function to any uncountable set is not continuous. The Continuum Hypothesis is necessary in this frame. Shinoda (\cite{shinoda}) showed that if both Martin's Axiom and the negation of the Continuum Hypothesis hold, then for every \(f\colon\mathbb R\to\mathbb R\) there exists an uncountable set \(Z\subset\mathbb R\) such that \(f\vert_Z\) is continuous. Following the notation from \cite{GMSS}, we say that a function \(f\colon\mathbb R\to\mathbb R\) is a \emph{Sierpi\'nski-Zygmund function} if it satisfies the condition in Sierpi\'nski-Zygmund's Theorem and the set of these functions is denoted by
$$\SZ(\mathbb R)=\{\,f\colon\mathbb R\to\mathbb R \,:\, f\text{ is a Sierpi\'nski-Zygmund function}\,\}.$$

In \cite{GMSS}*{Theorems 5.6 and 5.10} it was proved that the set $\SZ(\mathbb R)$ is $\mathfrak{c}^+$-lineable and, also, $\mathfrak{c}$-algebrable. As a consequence, assuming that $\mathfrak{c}^+ = 2^{\mathfrak{c}}$ (which follows, for instance, from the Generalized Continuum Hypothesis or GCH), \(\SZ(\mathbb R)\) is also \(2^\mathfrak{c}\)-lineable. Up to today it was not known whether any additional set-theoretical assumptions is needed or not in order to show the $2^\mathfrak{c}$-lineability of $\SZ(\mathbb R)$. In Theorem \ref{SZth} we shall improve this latter results from \cite{GMSS} by providing a new technique that shall show that the set of all Sierpi\'nski-Zygmund functions is actually $\kappa$-algebrable for some $\mathfrak{c}^+ \le \kappa \le 2^{\mathfrak{c}}$.

Let us recall that, besides its intrinsic mathematical interest, this class of Sierpi\'nski-Zygmund functions has attracted the interest of many mathematicians, especially in the area of Real Analysis, since its appearance at the beginning of the 20th century (see, e.g. \cite{X1,X2,X3,X4}). In \cite{X3} it was shown that there exists a Darboux function that is also in $\SZ(\mathbb{R})$ and that there exists a model of ZFC in which there are no such functions.

\noindent Let us begin by giving certain results that shall be needed throughout this section.

\begin{lemma}\label{l1}
Let $P$ be a non-zero polynomial in $n+1$ variables. Then the set $$\{x : \forall y \in \R^n, \quad P(x,y)=0\}$$ is finite.
\end{lemma}

\begin{proof}
Suppose that the set $\{x:\forall y\in\R^n\;\;P(x,y)=0\}$ is infinite. Let $y\in\R^n$. Then the polynomial $x\mapsto P(x,y)$ in one variable has infinitely many zeros, and therefore it vanishes for every $x$. Hence $P\equiv 0$.
\end{proof}

\begin{lemma}\label{l2}
Let $\mathcal{P}$ be a family of polynomials of cardinality less than $\mathfrak{c}$. Then there exists a set $Y=\{y_\xi:\xi<\mathfrak{c}\}$ such that $P(y_{\xi_1},y_{\xi_2},\ldots,y_{\xi_n})\neq 0$ for any $n\geq 1$, any polynomial $P\in\mathcal{P}$ in $n$ variables and any distinct ordinals $\xi_i<\mathfrak{c}$.
\end{lemma}

\begin{proof}
Let $\lambda=|\mathcal{P}|$.
By $\mathcal{P}_n$ we denote the set of polynomials from $\mathcal{P}$ in $n$ variables.
Without loss of generality we may assume that $\mathcal{P}$ is closed under taking permutations of variables, i.e. for any $n$, $P\in\mathcal{P}_n$ and any permutation $\sigma\in S_n$, the polynomial $Q$ defined by
$$
Q(x_1,x_2,\ldots,x_n)=P(x_{\sigma(1)},x_{\sigma(2)},\ldots,x_{\sigma(n)})
$$
is in $\mathcal{P}$.

We define $\{y_\xi:\xi<\mathfrak{c}\}$ by induction. Suppose that we have already defined  $Y_\alpha=\{y_\xi:\xi<\alpha\}$ such that
\begin{itemize}
\item[$(1)$] $\forall n<\min\{\omega,\alpha+1\}\forall P\in\mathcal{P}_n\forall \xi_1<\xi_2<\ldots<\xi_n<\alpha (P(y_{\xi_1},y_{\xi_2},\ldots,y_{\xi_n})\neq 0)$
\item[$(2)$] $\forall n\forall m<\min\{n,\alpha+1\}\forall P\in\mathcal{P}_n\forall \xi_1<\xi_2<\ldots<\xi_m<\alpha\exists y\in\R^{n-m}(P(y_{\xi_1},y_{\xi_2},\ldots,y_{\xi_m},y)\neq 0)$.
\end{itemize}
We find $y_\alpha\notin Y_\alpha$ such that
\begin{itemize}
\item[$(1')$] $\forall n<\min\{\omega,\alpha+1\}\forall P\in\mathcal{P}_{n+1}\forall \xi_1<\xi_2<\ldots<\xi_n<\alpha (P(y_{\xi_1},y_{\xi_2},\ldots,y_{\xi_n},y_\alpha)\neq 0)$
\item[$(2')$] $\forall n\forall m<\min\{n,\alpha+1\}\forall P\in\mathcal{P}_{n+1}\forall \xi_1<\xi_2<\ldots<\xi_m<\alpha\exists y\in\R^{n-m}(P(y_{\xi_1},y_{\xi_2},\ldots,y_{\xi_m},y_\alpha,y)\neq 0)$.
\end{itemize}

Note that, by $(2)$, for a given $P\in\mathcal{P}_{n+1}$ and $\xi_1<\ldots<\xi_n<\alpha$ the set $\{x:P(y_{\xi_1},\ldots,y_{\xi_n},x)=0\}$ is finite. Therefore the set
$$
W:=\{x:\exists n<\min\{\omega,\alpha+1\}\exists P\in\mathcal{P}_{n+1}\exists\xi_1<\ldots<\xi_n<\alpha(P(y_{\xi_1},\ldots,y_{\xi_n},x))=0\}
$$
is of cardinality less than $\max\{\omega,\alpha,\lambda\}<\mathfrak{c}$.

Now, by Lemma \ref{l1} and $(2)$, we have that for a given $P\in\mathcal{P}_{n+1}$, $m<\min\{n,\alpha+1\}$ and $\xi_1<\ldots<\xi_m<\alpha$ the set $\{x:P(y_{\xi_1},\ldots,y_{\xi_m},x,y)=0 $ for all $y\in\R^{n-m}\}$ is finite. Therefore the set
$$
X:=\{x:\exists n\exists m<\min\{n,\alpha+1\}\exists P\in\mathcal{P}_{n+1}\exists\xi_1<\ldots<\xi_n<\alpha
$$
$$
\forall y\in\R^{n-m}(P(y_{\xi_1},\ldots,y_{\xi_n},x,y))=0\}
$$
is of cardinality less than $\max\{\omega,\alpha,\lambda\}<\mathfrak{c}$.
Since $Y_\alpha$ is of cardinality less than $\mathfrak{c}$ we can find $y_\alpha$ which is not in $Y_\alpha\cup W\cup X$. This finishes the inductive proof.
\end{proof}

\begin{lemma}\label{l3}
Let $\mathcal{P}$ be a family of non-zero real polynomials with no constant term and let $X$ be a subset of $\R$ both of cardinality less than $\mathfrak{c}$. Then there exists set $Y=\{y_\xi:\xi<\mathfrak{c}\}$ such that $P(y_{\xi_1},y_{\xi_2},\ldots,y_{\xi_n})\notin X$ for any $n$, any polynomial $P\in\mathcal{P}$ and any distinct ordinals $\xi_i<\mathfrak{c}$.
\end{lemma}

\begin{proof}
Consider the family of polynomials $\mathcal{P}'=\{P-x:P\in\mathcal{P},x\in X\}$. Then using Lemma \ref{l2} for $\mathcal{P'}$ we easily obtain the assertion.
\end{proof}

Before stating and proving the main result of this section, let us recall the set theoretical notion of {\em almost disjoint sets}. Although this notion admits several variants depending on the framework we are working on, given $A$ and $B$ any two subsets of $\mathbb{R}$ of cardinality $\mathfrak{c}$ we shall say that $A$ and $B$ are almost disjoint if $|A \cap B| < \mathfrak{c}$  (see, e.g. \cite{almostdisjoint1}*{Section 1}). Let us recall that the existence of $2^{\mathfrak{c}}$ almost disjoint subsets of the real line each of them of cardinality $\mathfrak{c}$ follows, for instance, from Martin's Axiom (MA) and CH (see, e.g. \cite{K}). It is also known that, in general (see \cite{almostdisjoint1}), there exists an almost disjoint family of cardinality $\kappa^+$ on $\kappa$.

Now, with this at hand (and in the usual framework of ZFC) it is time to state and prove the main result of this section.

\begin{theorem}\label{SZth}
The set of Sierpi\'nski-Zygmund functions is strongly $\kappa$-algebrable, provided there exists a family of $\kappa$ almost disjoint subsets of $\mathfrak{c}$.
\end{theorem}

\begin{proof}
Let $\{g_\alpha:\alpha<\mathfrak{c}\}$ be a well-ordering of all Borel functions and write $\R=\{x_\alpha:\alpha<\mathfrak{c}\}$. Finally, let $\{P_\alpha:\alpha<\mathfrak{c}\}$ denote the set of all non-zero polynomials without constant term. Define, inductively, $\{Y_\alpha:\alpha<\mathfrak{c}\}$ a family of subsets of $\R$ each of which has cardinality $\mathfrak{c}$. At the stage $\alpha$ we use Lemma \ref{l3} for $X:=\{g_\lambda(x_\alpha):\lambda\leq\alpha\}$ and $\mathcal{P}:=\{P_\beta:\beta\leq\alpha\}$ to define $Y_\alpha$.  Next, let $Y_\alpha=\{y_\xi^\alpha:\xi<\mathfrak{c}\}$ and define $\mathcal{Y}=\prod_{\alpha<\mathfrak{c}}Y_\alpha$. Let $\{N_\zeta:\zeta<\kappa\}$ be a set of almost disjoint subsets of $\mathfrak{c}$ each of cardinality $\mathfrak{c}$. For any $\zeta<\kappa$ let $\{\zeta(\xi):\xi<\mathfrak{c}\}$ be an increasing enumeration of $N_\zeta$ and define $f_\zeta:\R\to\R$ by $f_\zeta(x_\alpha)=y^\alpha_{\zeta(\alpha)}$. It suffices to show that $\{f_\zeta:\zeta<\kappa\}$ is a set of free generators of an algebra consisting of Sierpi\'nski-Zygmund functions.. Let $\zeta_1<\zeta_2<\ldots<\zeta_n<\kappa$, $P_\beta$ be a polynomial in $n$ variables, $g_\gamma$ be a Borel function and $Z$ be any subset of $\R$ of cardinality $\mathfrak{c}$. Since $\{N_\zeta:\zeta<\kappa\}$ is an almost disjoint family of $\mathfrak{c}$, there is $\xi<\mathfrak{c}$ such that $N_{\zeta_1},N_{\zeta_2},\ldots,N_{\zeta_n}$ are disjoint above $\xi$ (i.e. $N_{\zeta_i}\cap(\xi,\mathfrak{c})$ are pairwise disjoint, where $(\xi,\mathfrak{c})$ is a set consisting of all ordinals between $\xi$ and $\mathfrak{c}$). Since $Z$ is of cardinality $\mathfrak{c}$, there is $\alpha<\mathfrak{c}$ with $\alpha>\max\{\beta,\gamma,\xi\}$ and $x_\alpha\in Z$. Since $\alpha$ is greater than $\xi$, then $f_{\zeta_1}(x_{\alpha}),f_{\zeta_2}(x_{\alpha}),\ldots,f_{\zeta_n}(x_{\alpha})$ are distinct points of $Y_\alpha$. Since $\alpha$ is greater than $\beta$ and $\gamma$, by construction $P_\beta(f_{\zeta_1},f_{\zeta_2},\ldots,f_{\zeta_n})$ differs from $g_{\gamma}$ at the point $x_\alpha\in Z$. Therefore $P_\beta(f_{\zeta_1},f_{\zeta_2},\ldots,f_{\zeta_n})$ is a Sierpi\'nski-Zygmund function and the result follows.
\end{proof}

Of course, we have the following consequences.

\begin{corollary}
If one of the following set-theoretical assumption holds 
\begin{itemize}
\item Martin's Axiom, or 
\item CH or,
\item $\mathfrak{c}^+=2^\mathfrak{c}$,
\end{itemize}
then the set of Sierpi\'nski-Zygmund functions is $2^{\mathfrak{c}}$-algebrable.
\end{corollary}

Also, and from the previous corollary, we can obtain (as particular cases) \cite{GMSS}*{Theorem 5.6} and \cite{GMSS}*{Theorem 5.10}. The following natural question now follows:

\begin{question}
Is it necessary to add any additional hypothesis to ZFC in order to obtain $2^{\mathfrak{c}}$-algebrability (or even $2^{\mathfrak{c}}$-lineability) of $\SZ(\mathbb{R})$?
\end{question}

\section{Perfectly everywhere surjective functions}
In the last years many authors have been interested in the study of different degrees of {\em surjectivity} of functions in $\mathbb{R}$ (see, e.g. \cite{AGS,ACPS,AS,G,GMSS}). Lebesgue \cites{Gel-Olm,Lebesgue} was probably the first to show an example of a function $f\colon\mathbb{R} \to \mathbb{R}$ with the property that on every non-trivial interval $I$,  $f(I) = \mathbb{R}$. In \cite{AGS} the authors proved that the set of these so called {\it everywhere surjective} functions is $2^{\mathfrak{c}}$-lineable, which is the \emph{best} possible result in terms of dimension. One could think that, in terms of surjectivity, these everywhere surjective functions are somewhat the {\it most pathological}. Well, actually (and as we shall see in this and the following section) even \emph{more pathological} surjective functions exist. For instance, there exists a function $f\colon \mathbb{R} \to \mathbb{R}$ such that for every perfect set $P$ and for every $r \in \mathbb{R}$ the set $\{\,x \in P: f(x) = r\,\}$ has cardinality $\mathfrak{c}$ (see \cite{Foran}*{pp.\ 124-125} and \cite{GMSS}*{Example 2.1}). The latter functions are called {\em perfectly everywhere surjective}. A perfectly everywhere surjective function actually attains every real value \emph{$\mathfrak{c}$ times} on any perfect set. Clearly, if $f$ is perfectly everywhere surjective then $f$ is everywhere surjective and the converse is false (see \cite{GMSS}*{Example 2.2}). In this section we shall consider complex functions $f:\C \to \C$ and we shall say that $f$ is perfectly everywhere surjective provided $f(P)=\C$ for every perfect subset $P$ of the complex plane. Recently, it was proved that:
\begin{theorem}[\cite{GMSS}]
The set of perfectly everywhere surjective functions on $\mathbb{C}$ is $\omega$-algebrable.
\end{theorem}
Here, our aim is to improve the above theorem by proving that the set of perfectly everywhere surjective functions on $\mathbb{C}$ is, actually, $2^{\mathfrak{c}}$-algebrable, which (in terms of cardinalities and dimensions) is the best possible result. From now on we shall denote
$$\mathcal{PES}(\C) = \{f: \C \rightarrow \C: f \text{ is perfectly everywhere surjective}\}.$$

Let us begin by recalling that a subset $B$ of $\C$ is called a Bernstein set if $B$ and its complement $B^c$ meet every perfect subset of $\C$. Since every perfect set $P$ contains $\mathfrak{c}$ many perfect subsets, then $P\cap B$ has cardinality $\mathfrak{c}$. Note that any polynomial $Q:\C\to\C$ is an entire function, and thus if $f$ is perfectly everywhere surjective then so is $Q(f)$; in particular this means that $f$ is a free generator in the algebra $\C^\C$ with the usual point-wise addition and product. The following lemma
 can be found in \cite{BG1}.
\begin{lemma}\label{l4}
Let $A$ be a commutative algebra with unit. Let $x\in A$ be a free generator (i.e. $x, x^2, x^3, \ldots$ are linearly independent) and let $P\neq 0$ be a polynomial in $n$ variables without non-zero constant term. Then at least one of the elements $P(x,1, \ldots,1)$, $P(1,x,1, \ldots,1)$, \ldots, $P(1, \ldots,1,x)$ is non-zero.
\end{lemma}

The following proposition shall be necessary in order to prove the main result of this section, and its proof follows the same spirit of that in \cite{GMSS}*{Example 2.1}.

\begin{proposition}\label{prop1}
Let $B \subset \C$ be a Bernstein set. There exist $f\in \mathcal{PES}(\C)$ which is identically 1 on the complement of $B$.
\end{proposition}

\begin{proof}
First of all let us recall that if $P$ is a perfect set then $P \cap B$ has cardinality $\mathfrak{c}$. Now, the cardinality of the collection of pairs $(P \cap B,y)$ where $P$ is a perfect set and $y \in \mathbb{C}$ is $\mathfrak{c}$. Let $\{ A_{\alpha} : \alpha < \mathfrak{c} \}$ be a well-ordering of this set and notice that each perfect set $P$ occurs $\mathfrak{c}$ times at a first part of a pair $(P\cap B,y)$. \newline
Assume that $A_0 = (P_0\cap B,y_0)$, choose $x_0 \in P_0\cap B$ and define $f(x_0) = y_0$. Let us assume now that, for each $\beta < \alpha$, $A_{\beta} = (P_{\beta}\cap B,y_{\beta})$ and that a point $x_{\beta} \in P_{\beta}\cap B$ has been chosen in such a way that for every $\gamma < \beta$, $x_{\beta} \neq x_{\gamma}$ and with $f(x_{\beta}) = y_{\beta}$. Next, let $A_{\alpha} = (P_{\alpha}\cap B, y_{\alpha})$. Since the cardinality of the set $\{x_{\beta} : \beta < \alpha\}$ is less than $\mathfrak{c}$, there exists $x_{\alpha} \in (P_{\alpha}\cap B) \setminus \{x_{\beta}: \beta < \alpha\}$. Now, define $f(x_{\alpha}) = y_{\alpha}$. Thus, for each $\alpha < \mathfrak{c}$, a point has been chosen so that for $\beta < \alpha$, $x_{\alpha} \neq x_{\beta}$ and $f(x_{\alpha})=y_{\alpha}$. If $x \notin \{x_{\alpha} :  \alpha< \mathfrak{c}\}$, we define $f(x) = 1$. By doing this, $f$ is now defined on $\mathbb{C}$. \newline
Now, if $P$ is any perfect set and $y \in \mathbb{C}$, then there exists $\alpha < \mathfrak{c}$ such that $(P\cap B,y) = A_{\alpha}$. Therefore there exists $x = x_{\alpha} \in P\cap B$ with $f(x) = y$, and thus $f(P) \supset f(P\cap B) = \mathbb{C}$. Also, and by construction, it can be checked that $f$ is identically $1$ on $B^{c}$.
\end{proof}

Before carrying on, let us first recall the notion of {\em independence} in abstract set theory (see, e.g. \cite{BF}).

\begin{definition}
A family $(A_s)$ of subsets of $\kappa$ (or any fixed set of cardinality $\kappa$) is called independent if the set $A_{s_1}^{i_1} \cap \ldots \cap A_{s_n}^{i_n}$ has cardinality $\kappa$ for any $n\in \mathbb{N}$, distinct $s_1, \ldots, s_n$ and any $i_k \in \{0,1\}$, where we denote $A^0=A$ and $A^1$ stands for the complement of $A$ for any set $A$.
\end{definition}

From a general result due to Balcar and Fran{\v{e}}k (see \cite{BF}), it follows that for any set of cardinality $\kappa$, there exists an independent family of its subsets of cardinality the largest possible, i.e. $2^\kappa$. This shall be used in the main result of this section, that we state and prove  now.

\begin{theorem}\label{PESth}
The set $\mathcal{PES}(\C)$ is strongly $2^\mathfrak{c}$-algebrable.
\end{theorem}

\begin{proof}
Let $\{A_\zeta:\zeta<2^\mathfrak{c}\}$ be an independent family of subsets of $\mathfrak{c}$. Let $\{B_\alpha:\alpha<\mathfrak{c}\}$ be a family of pairwise disjoint Bernstein subsets of $\C$. Let $g_\alpha$ be a perfectly everywhere surjective function as in Proposition \ref{prop1} (for $B_\alpha$). Let us define $f_\zeta: \C \to \C$ as follows:
$$f_\zeta(x)=\left\{
\begin{array}
[c]{cl}
g_\alpha(x) & \mathrm{if~} x\in B_\alpha \, \mathrm{ and } \, \alpha\in A_\zeta,\\
1 & \mathrm{otherwise}.
\end{array}
\right.$$
We shall prove that $\{f_\zeta:\zeta<2^\mathfrak{c}\}$ is a family of free generators.  Let $\zeta_1<\zeta_2< \ldots <\zeta_n<2^\mathfrak{c}$ and let $P$ be a polynomial in $n$ variables. It suffices to show that $P(f_{\zeta_1},f_{\zeta_2},\ldots,f_{\zeta_n})$ is in $\mathcal{PES}(\C)$. Since $\{A_\zeta: \zeta < 2^\mathfrak{c}\}$ is independent, there is $\alpha\in A_{\zeta_1}\setminus\bigcup_{i=2}^{n} A_{\zeta_i}$. Then $P(f_{\zeta_1},f_{\zeta_2},\ldots, f_{\zeta_n})|_{B_\alpha}=P(g_\alpha|_{B_\alpha},1, \ldots,1)$. As we mentioned earlier in this section $g_\alpha$ is a free generator. Therefore, by Lemma \ref{l4}, one of the elements $P(f_{\zeta_1},1, \ldots,1)$, $P(1,f_{\zeta_1},1, \ldots,1)$, $\ldots$,$P(1, \ldots,1,f_{\zeta_1})$ is non-zero on $B_\alpha$, and therefore is in $\mathcal{PES}(\C)$. Without loss of generality we may assume that $P(1, \ldots,1,f_{\zeta_1})$ is non-zero on $B_\alpha$.

Now, we shall show that $P(1, \ldots,1,f_{\zeta_n})$ is non-zero and, therefore belongs to $\mathcal{PES}(\C)$. Since $\{A_\zeta:\zeta<2^\mathfrak{c}\}$ is independent, there exists $\beta\in A_{\zeta_n}\setminus\bigcup_{i=1}^{n-1}A_{\zeta_i}$. Then $f_{\zeta_n}$ is equal to $g_\beta$ over $B_\beta$. Notice also that the fibres $g_\alpha^{-1}(x)$ and $g_\beta^{-1}(x)$ have cardinality $\mathfrak{c}$ for every $x \in \C$. Thus, there is a bijection $h:B_\alpha \leftrightarrow B_\beta$ such that $h$ maps $g_\alpha^{-1}(x)$ onto $g_\beta^{-1}(x)$ for every $x\in\C$. Then $f_{\zeta_1}|_{B_\alpha}=f_{\zeta_n}|_{B_\beta}\circ h$. Hence $$P(1, \ldots,1,f_{\zeta_1}|_{B_\alpha})=P(1, \ldots,1,f_{\zeta_n}|_{B_\beta}\circ h)$$ is non-zero. Thus, $P(1, \ldots,1,f_{\zeta_n}|_{B_\beta})$ is non-zero and the result follows.
\end{proof}

\section{Nowhere continuous Darboux functions}
As it is known (at an undergraduate level) if a function \(f\in \mathbb{R}^\mathbb{R}\) is continuous, then it maps connected sets to connected sets (in other words, it is a Darboux function). Very recently (see \cite{GMPS_LAA} and \cite{monthly}*{Theorem 1}) it was constructed a $2^{\mathfrak{c}}$-dimensional linear space of functions in $\mathbb{R}^{\mathbb{R}}$ every non-zero element of which is nowhere continuous and Darboux. In this section we are interested in, again, moving from lineability to algebrability, trying to improve the result just mentioned. In order to do that, we shall need to define several notions that shall be of help throughout this section.

Of course, and although there is no need in saying it here, if $f \in \pes$ then $f$ is nowhere continuous. As we mentioned in the previous section, a function \(f\in \mathbb{R}^\mathbb{R}\) is called {\em everywhere surjective} (denoted $f \in \es$, \cite{AGS}) if for every non-void interval $I$ of $\mathbb{R}$ we have $f(I) = \mathbb{R}$. It is clear that $\pes \subset \es$ (see \cite{GMSS}).

If \(f\in \mathbb{R}^\mathbb{R}\), we say that \(f \in \jones\) (\(f\) is a Jones function) if for every closed set \(K \subset \mathbb R^2\) with uncountable projection on the \(x\)-axis, we have \(f \cap K \neq \varnothing\) (see \cite{jones}). As it was proved in \cite{GMS_add}, the following holds $$\jones \subset \pes \subset \es.$$
Interestingly enough, it is known (\cite{GMSS,GMS_add}) that neither perfectly everywhere surjective functions nor Jones functions are measurable and, in contrast with Section 2 in this paper, we also have that  (\cite{GMS_add}*{Corollary 3.8}) $\SZ(\mathbb R) \cap \jones = \varnothing$.

Throughout this section, we shall identify a function from \(\mathbb R\) to \(\mathbb R\) with its graphic, that is, we shall consider that function as a subset of \(\mathbb R^2\). Nevertheless, sometimes it is more convenient the opposite procedure, and we shall consider also a set \(E\in \mathbb R^2\) as a correspondence between \(\mathbb R\) and \(\mathbb R\). Let us also write
	$$\dom (E) = \{ \,x\in\mathbb R:\text{there exists }y \in \mathbb R \text{ such that }(x,y) \in E \, \},$$
which is nothing but a projection of $E$ on the first coordinate.
Although the set of Jones functions is $2^{\mathfrak{c}}$-lineable (see \cite{G,GMS_add}), it is easy to see that it is not algebrable. Indeed, if \(f\in J(\mathbb R) \), then \(f^2\) is not even surjective. In \cite{GMPS_preprint} the authors defined a dual concept to that of Jones functions but on the complex plane and they showed that (in the complex setting) this set is actually algebrable.

The following result, although of independent interest, provides a relatively easy way to see that the set of Jones functions is $\mathfrak{c}$-lineable (see also \cite{G}*{Theorem 2.2} for the best possible result in this direction).

\begin{proposition}\label{compo}
Let \(f\in\jones\) and \(\varphi\in \mathcal C(\mathbb R)\) surjective. Then \(\varphi\circ f\in \jones\).
\end{proposition}

\begin{proof}
Let \(\Phi\colon\mathbb R^2\to\mathbb R^2\) defined by \(\Phi(z,w)=(z,\varphi(w))\). Obviously \(\Phi\) is a continuous function. Let \(K\subset\mathbb R^2\) a closed set with \(\dom (K)\) uncountable. The set \(\Phi^{-1}(K)\) is clearly closed. Besides, we have \(\dom((\Phi^{-1}(K)))\supset\dom (K)\). Indeed, if \(z\in\dom (K)\), there exists \(w\) such that \((z,w)\in K\). As \(\varphi\) is surjective, there exists \(w'\) such that \(\varphi(w')=w\). So, \(\Phi(z,w')=(z,\varphi(w'))\in K\). That is, \((z,w')\in \Phi^{-1}(K)\), and so \(z\in\dom((\Phi^{-1}(K)))\). Therefore, \(\dom((\Phi^{-1}(K)))\) is uncountable, and hence, since \(f \in \jones\), we have \(f\cap \Phi^{-1}(K)\neq\varnothing\). Let \((z,w)\in f\cap \Phi^{-1}(K)\). Then we have \(w=f(z)\) and \((z,\varphi(w))=\Phi(z,w)\in K\). That means that \((z,(\varphi\circ f)(z))\in K\), that is, \((\varphi\circ f)\cap K\neq\varnothing\).
\end{proof}

As we said right before the previous result, and since dim$(\mathcal{C}(\mathbb{R})) = \mathfrak{c}$, it suffices to take $f \in \jones$ and the linear space
$$\left\{ \varphi \circ f \,:\, \varphi\in \mathcal C(\mathbb R)\right\}$$
to achieve the $\mathfrak{c}$-lineability of $\jones$.

The following result shall be useful in the proof of the main result of this section.

\begin{proposition}\label{connected}
Let \(f\in\es\) and \(\varphi \in \mathcal C(\mathbb R)\). Then \(\varphi\circ f\) is a Darboux function.
\end{proposition}

\begin{proof}
From Proposition \ref{compo} the result immediately follows if $\varphi$ is surjective. In general, take $I$ any connected subset of $\mathbb{R}$, that is, an interval. Take any $f \in \es$, then (by definition) $f(I) = \mathbb{R}$. Then, $\varphi (f(I))$ is an interval and, therefore, connected.
\end{proof}

It can be seen (see \cite{AGS,GMSS,GMS_add}) that if $p$ is a non-constant polynomial in $\mathbb{R}$ and if $f \in \es$, then $p\circ f$ is nowhere continuous. This last remark, together with the following corollary to Proposition \ref{connected}, shall be useful in the main result of this section.

\begin{corollary}\label{connected_cor}
Let \(f\in\es\) and $p$ any non-constant polynomial in $\mathbb{R}$. Then \(p\circ f\) is nowhere continuous and Darboux.
\end{corollary}

\begin{remark}
One could think that by means of Corollary \ref{connected_cor} we could easily obtain algebrability from the fact that the ring of polynomials in $\mathbb{R}$ is itself an algebra. Unfortunately this issue is more complicated than it might seem at first sight. The reason is that the ring of polynomials in $\mathbb{R}$ does not contain any infinitely generated algebra. Indeed, in $\mathbb{K} [x]$ ($\mathbb{K} = \mathbb{R}$ or $\mathbb{C}$), suppose  that $S$ is a subalgebra that contains $\mathbb{K}.$ Let $P$ be a non-constant polynomial in $S$ and let $T$ be the subalgebra of $S$ generated by $1$ and $P$. Then $x$ is algebraic over $T$, so $\mathbb{K}[x]$ is a finitely generated $T$-module. Note that $T$ is Noetherian and $S$ is a $T$-submodule of $\mathbb{K}[x]$, so $S$ is a finitely generated $T$-module. Hence $S$ is a finitely generated algebra. Obviously, this algebra has infinite dimension ($\omega$) as a linear space. For this reason we need a totally different argument to tackle this problem of algebrability.
\end{remark}

The following lemma, which is a slight variation of \cite{G}*{Lemma 2.1} shall be needed throughout this section (its proof follows the same idea as in that of \cite{G}*{Lemma 2.1}, and we spare the details of it to the interested reader).

\begin{lemma}
Let \(B\subset\mathbb R\) be a Bernstein set. There exists \(f:\mathbb R\to\mathbb R\) such that
\begin{enumerate}
\item \(f\vert_B\cap K\neq\varnothing\) for every closed set \(K \subset \mathbb R^2\) such that \(\dom (K)\) is uncountable, and
\item \(f(x)=0\) for every \(x\notin B\).
\end{enumerate}
\end{lemma}

A classical result states that \(\mathbb R\) can be decomposed as the disjoint union of ``\(\mathfrak c\) many'' Bernstein sets. Thus, consider such a decomposition \(\mathbb R = \bigcup_{i\in\mathbb R} B_i\), where \(B_i\) is a
Bernstein set for every \(i\in\mathbb R\) and the family $\{B_i: i \in \mathbb{R} \}$ is disjoint. Now, use the previous
lemma to obtain, for every \(i\in\mathbb R\), a function \(f_i\colon\mathbb R\to\mathbb R\) such that \(f_i\vert_{B_i}\cap K\neq\varnothing\) for every closed set \(K\subset\mathbb R^2\) with \(\dom (K)\) uncountable. Assume, also, that \(f_i=0\) outside \(B_i\). Notice that (by construction)
    \begin{equation}\label{orthogonal}
    f_i\cdot f_j=0\quad\text{if}\quad i\neq j.
    \end{equation}
Now, let \(\mathcal B\) be the algebra generated by \(\{ f_i \,:\, i \in \mathbb R \}\), that is, by \eqref{orthogonal}, the linear span of the set \(\{\,f_i^p:i\in\mathbb R,p\in\mathbb N\,\}\).

At this point it is worth mentioning that, by considering the elements in the algebra  \(\mathcal B\), we have lost the surjectivity (and, thus, the everywhere surjectivity and the properties enjoyed by Jones functions) since, for instance, the square of any of the $f_i$'s is not even surjective. On the other hand, taking into account that $\jones \subset \es$ and Corollary \ref{connected_cor} together with that (by construction) the $f_i$'s have disjoint supports it can be seen that every non-zero element of \(\mathcal B\) is nowhere continuous and maps connected sets to connected sets (actually, to unbounded intervals!).

Clearly, the family \(\{\,f_i: i\in\mathbb R\,\}\) is a generator system of \(\mathcal B\). Let us now see that this generator system is actually  minimal. Suppose otherwise, then there would exist \(f_{i_j}\), \(j \in \{0,1,2,\dotsc,n \}\) (with \(i_j\neq i_k\) if \(j\neq k\)) and a polynomial \(P(x_1,x_2,\dotsc,x_n)\) such that
$$f_{i_0}=P(f_{i_1},f_{i_2},\dotsc,f_{i_n}).$$
By \eqref{orthogonal}, there exist polynomials \(p_1(x)\), \(p_2(x)\), \dots, \(p_n(x)\), such that
    $$f_{i_0}=P(f_{i_1},f_{i_2},\dotsc,f_{i_n})=p_1(f_{i_1})+p_2(f_{i_2})+\dotsb+p_n(f_{i_n}).$$
But, for every \(j=1,2,\dotsc,n\), we have that (by construction) \(p_j(f_{i_j})\) is a constant function on \(B_{i_0}\). This is a contradiction, since \(f_{i_0}\) is surjective on \(B_{i_0}\). Thus, we have proved the following result:

\begin{theorem}\label{3th}
The set of nowhere continuous Darboux functions is $\mathfrak{c}$-algebrable.
\end{theorem}

\noindent {\bf Acknowledgements.} The authors wish to thank Prof. Dr. J. L. G\'amez-Merino for fruitful conversations and comments on this paper.



\begin{bibdiv}
\begin{biblist}

\bib{ACPS}{article}{
   author={Aron, Richard M.},
   author={Conejero, Jos{\'e} A.},
   author={Peris, Alfredo},
   author={Seoane-Sep{\'u}lveda, Juan B.},
   title={Uncountably generated algebras of everywhere surjective functions},
   journal={Bull. Belg. Math. Soc. Simon Stevin},
   volume={17},
   date={2010},
   number={3},
   pages={571--575},
}

\bib{AGM}{article}{
    author={Aron, R. M.},
    author={Garc{\'{\i}}a, D.},
    author={Maestre, M.},
     TITLE = {Linearity in non-linear problems},
   JOURNAL = {RACSAM Rev. R. Acad. Cienc. Exactas F\'{\i}s. Nat. Ser. A Mat.},
    VOLUME = {95},
      YEAR = {2001},
    NUMBER = {1},
     PAGES = {7--12},
}

\bib{topology}{article}{
   author={Aron, R. M.},
   author={Garc\'{i}a-Pacheco, F. J.},
   author={P\'{e}rez-Garc\'{i}a, D.},
   author={Seoane-Sep\'{u}lveda, J. B.},
   title={On dense-lineability of sets of functions on $\mathbb{R}$},
   journal={Topology},
   volume={48},
   date={2009},
   pages={149--156},
}

\bib{AGS}{article}{
   author={Aron, R. M.},
   author={Gurariy, V. I.},
   author={Seoane-Sep\'{u}lveda, J. B.},
   title={Lineability and spaceability of sets of functions on $\Bbb R$},
   journal={Proc. Amer. Math. Soc.},
   volume={133},
   date={2005},
   number={3},
   pages={795--803},
   issn={0002-9939},
}

\bib{APS}{article}{
    AUTHOR = {Aron, R. M.}
    author={P{\'e}rez-Garc{\'{\i}}a, D.},
    author={Seoane-Sep{\'u}lveda, J. B.},
     TITLE = {Algebrability of the set of non-convergent {F}ourier series},
   JOURNAL = {Studia Math.},
  FJOURNAL = {Studia Mathematica},
    VOLUME = {175},
      YEAR = {2006},
    NUMBER = {1},
     PAGES = {83--90},
}
	
\bib{AS}{article}{
   author={Aron, R. M.},
   author={Seoane-Sep\'{u}lveda, J. B.},
   title={Algebrability of the set of everywhere surjective functions on
   $\Bbb C$},
   journal={Bull. Belg. Math. Soc. Simon Stevin},
   volume={14},
   date={2007},
   number={1},
   pages={25--31},
   issn={1370-1444},
}

\bib{BF}{article}{
   author={Balcar, B.},
   author={Fran{\v{e}}k, F.},
   title={Independent families in complete Boolean algebras},
   journal={Trans. Amer. Math. Soc.},
   volume={274},
   date={1982},
   number={2},
   pages={607--618},
}

\bib{X3}{article}{
   author={Balcerzak, Marek},
   author={Ciesielski, Krzysztof},
   author={Natkaniec, Tomasz},
   title={Sierpi\'nski-Zygmund functions that are Darboux, almost
   continuous, or have a perfect road},
   journal={Arch. Math. Logic},
   volume={37},
   date={1997},
   number={1},
   pages={29--35},
}

\bib{BG1}{article}{
   author={Bartoszewicz, A.},
   author={G\l \c ab, S.},
   title={Strong algebrability of sets of sequences and functions},
   journal={Proc. Amer. Math. Soc.},
   status={In Press},
}

\bib{BG2}{article}{
   author={Bartoszewicz, A.},
   author={G\l \c ab, S.},
   title={Algebrability of conditionally convergent series with Cauchy product},
   journal={J. Math. Anal. Appl.},
   volume={385},
   date={2012},
   pages={693--697},
}

\bib{BGP_LAA}{article}{
   author={Bartoszewicz, A.},
   author={G\l \c ab, S.},
   author={Poreda, T.},
   title={On algebrability of nonabsolutely convergent series},
   journal={Linear Algebra Appl.},
   volume={435},
   year={2011},
   pages={1025--1028},
}

\bib{BQ}{article}{
    AUTHOR = {Bayart, F.},
    author = {Quarta, L.},
     TITLE = {Algebras in sets of queer functions},
   JOURNAL = {Israel J. Math.},
    VOLUME = {158},
      YEAR = {2007},
     PAGES = {285--296},
      ISSN = {0021-2172},
}

\bib{B}{article}{
   author={Bernal-Gonz\'{a}lez, L.},
   title={Dense-lineability in spaces of continuous functions},
   journal={Proc. Amer. Math. Soc.},
   volume={136},
   date={2008},
   number={9},
   pages={3163--3169},
   issn={0002-9939},
}

\bib{Blumberg}{article}{
   author={Blumberg, H.},
   title={New properties of all real functions},
   journal={Trans. Amer. Math. Soc.},
   volume={82},
   date={1922},
   pages={53--61},
   isbn={3-540-16474-X},
}

\bib{BDP}{article}{
   author={Botelho, G.},
   author={Diniz, D.},
   author={Pellegrino, D.},
   title={Lineability of the set of bounded linear non-absolutely summing
   operators},
   journal={J. Math. Anal. Appl.},
   volume={357},
   date={2009},
   number={1},
   pages={171--175},
}

\bib{BMP}{article}{
     author={Botelho, G.},
     author={Matos, M.},
     author={Pellegrino, D.},
     title={Lineability of summing sets of homogeneous polynomials},
    journal={Linear Multilinear Algebra},
   volume={58},
   date={2010},
   number={1-2},
   pages={61--74},
}

 \bib{X1}{article}{
   author={Ciesielski, Krzysztof},
   author={Natkaniec, Tomasz},
   title={Algebraic properties of the class of Sierpi\'nski-Zygmund
   functions},
   journal={Topology Appl.},
   volume={79},
   date={1997},
   number={1},
   pages={75--99},
}

\bib{X4}{article}{
   author={Ciesielski, Krzysztof},
   author={Natkaniec, Tomasz},
   title={On Sierpi\'nski-Zygmund bijections and their inverses},
   journal={Topology Proc.},
   volume={22},
   date={1997},
   number={Spring},
   pages={155--164},
}

\bib{almostdisjoint1}{article}{
   author={Erd{\H{o}}s, P.},
   author={Hajnal, A.},
   author={Milner, E. C.},
   title={On sets of almost disjoint subsets of a set},
   journal={Acta Math. Acad. Sci. Hungar},
   volume={19},
   date={1968},
   pages={209--218},
}

\bib{Foran}{book}{
   author={Foran, James},
   title={Fundamentals of real analysis},
   series={Monographs and Textbooks in Pure and Applied Mathematics},
   volume={144},
   publisher={Marcel Dekker Inc.},
   place={New York},
   date={1991},
   pages={xiv+473},
   isbn={0-8247-8453-7},
}

\bib{G}{article}{
   author={G\'{a}mez-Merino, J. L.},
   TITLE = {Large algebraic structures inside the set of surjective functions},
   journal={Bull. Belg. Math. Soc. Simon Stevin},
   volume={18},
   date={2011},
   pages={297--300},
}

\bib{GMPS_LAA}{article}{
   author={G\'{a}mez-Merino, J. L.},
   author={Mu\~{n}oz-Fern\'{a}ndez, G. A.},
   author={Pellegrino, D.},
   author={Seoane-Sep{\'u}lveda, J. B.},
   title={Bounded and unbounded polynomials and multilinear forms: Characterizing continuity},
   journal={Linear Algebra Appl.},
   doi={10.1016/j.laa.2011.06.050},
}

\bib{GMPS_preprint}{article}{
   author={G\'{a}mez-Merino, J. L.},
   author={Mu\~{n}oz-Fern\'{a}ndez, G. A.},
   author={Pellegrino, D.},
   author={Seoane-Sep{\'u}lveda, J. B.},
   title={Lineability and algebrability in subsets of complex functions},
   status={Preprint},
}

\bib{GMSS}{article}{
    author={G\'{a}mez-Merino, J. L.},
    author={Mu\~{n}oz-Fern\'{a}ndez, G. A.},
    author={S\'{a}nchez, V. M.},
    author={Seoane-Sep\'{u}lveda, J. B.},
    title = {Sierpi\'nski-Zygmund functions and other problems on lineability},
    journal = {Proc. Amer. Math. Soc.},
    volume={138},
    date={2010},
    number={11},
    pages={3863--3876},
}

\bib{GMS_add}{article}{
   author={G\'{a}mez-Merino, Jos{\'e} L.},
   author={Mu{\~n}oz-Fern{\'a}ndez, Gustavo A.},
   author={Seoane-Sep{\'u}lveda, Juan B.},
   title={Lineability and additivity in $\Bbb R^{\Bbb R}$},
   journal={J. Math. Anal. Appl.},
   volume={369},
   date={2010},
   number={1},
   pages={265--272},
}

\bib{monthly}{article}{
author={G\'{a}mez-Merino, J. L.},
author={Mu\~{n}oz-Fern\'{a}ndez, G. A.},
author={Seoane-Sep\'{u}lveda, J. B.},
title={A characterization of continuity revisited},
journal={Amer. Math. Monthly},
VOLUME = {118},
YEAR = {2011},
NUMBER = {2},
PAGES = {167--170},
}

\bib{GGMS_MN}{article}{
   author={Garc{\'{\i}}a, D.},
   author={Grecu, B. C.},
   author={Maestre, M.},
   author={Seoane-Sep{\'u}lveda, J. B.},
   title={Infinite dimensional Banach spaces of functions with nonlinear
   properties},
   journal={Math. Nachr.},
   volume={283},
   date={2010},
   number={5},
   pages={712--720},
}

\bib{Taiwan}{article}{
   author={Garc{\'{\i}}a-Pacheco, F. J.},
   author={Mart{\'{\i}}n, M.},
   author={Seoane-Sep{\'u}lveda, J. B.},
   title={Lineability, spaceability, and algebrability of certain subsets of
   function spaces},
   journal={Taiwanese J. Math.},
   volume={13},
   date={2009},
   number={4},
   pages={1257--1269},
}

\bib{Gel-Olm}{book}{
   author={Gelbaum, Bernard R.},
   author={Olmsted, John M. H.},
   title={Counterexamples in analysis},
   series={The Mathesis Series},
   publisher={Holden-Day Inc.},
   place={San Francisco},
   date={1964},
   pages={xxiv+194},
}

\bib{G1}{article}{
   author={Gurariy, V. I.},
   title={Subspaces and bases in spaces of continuous functions (Russian)},
   journal={Dokl. Akad. Nauk SSSR},
   volume={167},
   date={1966},
   pages={971--973},
}

\bib{G2}{article}{
   author={Gurariy, V. I.},
   title={Linear spaces composed of nondifferentiable functions},
   journal={C.R. Acad. Bulgare Sci.},
   volume={44},
   date={1991},
   pages={13--16},
}

\bib{H}{article}{
   author={Hencl, Stanislav},
   title={Isometrical embeddings of separable Banach spaces into the set of
   nowhere approximatively differentiable and nowhere H\"older functions},
   journal={Proc. Amer. Math. Soc.},
   volume={128},
   date={2000},
   number={12},
   pages={3505--3511},
}

\bib{jones}{article}{
   author={Jones, F. B.},
   title={Connected and disconnected plane sets and the functional equation
   \(f(x)+f(y)=f(x+y)\)},
   journal={Bull. Amer. Math. Soc.},
   volume={48},
   date={1942},
   pages={115--120},
   issn={0002-9904},
}

\bib{Kharazishvili}{book}{
   author={Kharazishvili, A. B.},
   title={Strange functions in real analysis},
   series={Pure and Applied Mathematics},
   volume={272},
   edition={2},
   publisher={Chapman \& Hall/CRC},
   place={Boca Raton, Florida},
   date={2006},
   pages={xii+415},
}

\bib{K}{book}{
   author={Kunen, Kenneth},
   title={Set theory},
   series={Studies in Logic and the Foundations of Mathematics},
   volume={102},
   note={An introduction to independence proofs;
   Reprint of the 1980 original},
   publisher={North-Holland Publishing Co.},
   place={Amsterdam},
   date={1983},
   pages={xvi+313},
}

\bib{Lebesgue}{book}{
   author={Lebesgue, H.},
   title={Le\c{c}ons sur l'int\'{e}gration et la recherche des fonctions primitives},
   publisher={ChaGauthier-Willars},
   date={1904},
}

\bib{MPPS_LAA}{article}{
   author={Mu{\~n}oz-Fern{\'a}ndez, G. A.},
   author={Palmberg, N.},
   author={Puglisi, D.},
   author={Seoane-Sep{\'u}lveda, J. B.},
   title={Lineability in subsets of measure and function spaces},
   journal={Linear Algebra Appl.},
   volume={428},
   date={2008},
   number={11-12},
   pages={2805--2812},
}

\bib{X2}{article}{
   author={P{\l}otka, Krzysztof},
   title={Sum of Sierpi\'nski-Zygmund and Darboux like functions},
   journal={Topology Appl.},
   volume={122},
   date={2002},
   number={3},
   pages={547--564},
}

\bib{shinoda}{article}{
   author={Shinoda, Juichi},
   title={Some consequences of Martin's axiom and the negation of the
   continuum hypothesis},
   journal={Nagoya Math.~J.},
   volume={49},
   date={1973},
   pages={117--125},
   issn={0027-7630},
}

\bib{juanksu}{book}{
   author={Seoane-Sep\'{u}lveda, Juan B.},
   title={Chaos and lineability of pathological phenomena in analysis},
   note={Thesis (Ph.D.)--Kent State University},
   publisher={ProQuest LLC, Ann Arbor, MI},
   date={2006},
   pages={139},
   isbn={978-0542-78798-0},
}

\bib{SZ}{article}{
   author={Sierpi\'nski, W.},
   author={Zygmund, A.},
   title={Sur une fonction qui est discontinue sur tout ensemble de puissance du continu},
   journal={Fund. Math.},
   volume={4},
   date={1923},
   pages={316--318},
}

\end{biblist}
\end{bibdiv}
\end{document}